\documentclass[11pt, reqno]{amsart}
\usepackage{amssymb,amsmath,epsfig,mathrsfs, enumerate, xparse, mathtools}
\usepackage[french, english]{babel}
\usepackage{lipsum}
\usepackage[pagebackref,colorlinks=true,linkcolor=blue,citecolor=blue]{hyperref}

\usepackage{amsmath}
\usepackage{amsfonts}
\usepackage{amssymb}
\usepackage{amsthm}
\usepackage{epsfig,graphics,mathrsfs}
\usepackage{graphicx}
\usepackage{dsfont}

\usepackage[usenames, dvipsnames]{color}

\usepackage[pagewise]{lineno}
\usepackage[nodisplayskipstretch]{setspace}
\usepackage{graphicx}\usepackage[normalem]{ulem}
\usepackage{fancyhdr}
\usepackage[margin=2.5cm]{geometry}
\usepackage{hyperref}
\hypersetup{
	colorlinks   = true,
	urlcolor     = blue,
	linkcolor    = blue,
	citecolor   = red ,
	bookmarksopen=true
}

\theoremstyle{plain}
\newtheorem{thrm}{Theorem}[section]
\newtheorem{lemma}[thrm]{Lemma}

\usepackage[usenames]{color}

\allowdisplaybreaks
\begin{document}
	\newcommand{\sn}{\mathbb{S}^{n-1}}
	\newcommand{\SL}{\mathcal L^{1,p}( D)}
	\newcommand{\Lp}{L^p( Dega)}
	\newcommand{\CO}{C^\infty_0( \Omega)}
	\newcommand{\Rn}{\mathbb R^n}
	\newcommand{\Rm}{\mathbb R^m}
	\newcommand{\R}{\mathbb R}
	\newcommand{\Om}{\Omega}
	\newcommand{\Hn}{\mathbb H^n}
	\newcommand{\A}{\alpha }
	\newcommand{\B}{\beta}
	\newcommand{\eps}{\ve}
	\newcommand{\BVX}{BV_X(\Omega)}
	\newcommand{\p}{\partial}
	\newcommand{\IO}{\int_\Omega}
	\newcommand{\bG}{\boldsymbol{G}}
	\newcommand{\bg}{\mathfrak g}
	\newcommand{\bz}{\mathfrak z}
	\newcommand{\bv}{\mathfrak v}
	\newcommand{\Bux}{\mbox{Box}}
	\newcommand{\e}{\ve}
	\newcommand{\X}{\mathcal X}
	\newcommand{\Y}{\mathcal Y}
	\newcommand{\Z}{\mathcal Z}
	\newcommand{\I}{\mathcal I}
	\newcommand{\la}{\lambda}
	\newcommand{\vf}{\varphi}
	\newcommand{\rhh}{|\nabla_H \rho|}
	\newcommand{\Ba}{\mathcal{B}_\beta}
	\newcommand{\Za}{Z_\beta}
	\newcommand{\ra}{\rho_\beta}
	\newcommand{\n}{\nabla}
	\newcommand{\vt}{\vartheta}
	\newcommand{\its}{\int_{\{y=0\}}}
	\newcommand{\py}{\partial_y^a}
\newcommand{\sa}{\langle}
\newcommand{\da}{\rangle}
\newcommand{\mi}{\mathscr I}
\newcommand{\F}{\mathscr F}
	
	\numberwithin{equation}{section}

	\newcommand{\RN} {\mathbb{R}^N}
	\newcommand{\Sob}{S^{1,p}(\Omega)}
	\newcommand{\Dxk}{\frac{\partial}{\partial x_k}}
	\newcommand{\Co}{C^\infty_0(\Omega)}
	\newcommand{\Je}{J_\ve}
	\newcommand{\beq}{\begin{equation}}
		\newcommand{\bea}[1]{\begin{array}{#1} }
			\newcommand{\eeq}{ \end{equation}}
		\newcommand{\ea}{ \end{array}}
	\newcommand{\eh}{\ve h}
	\newcommand{\Dxi}{\frac{\partial}{\partial x_{i}}}
	\newcommand{\Dyi}{\frac{\partial}{\partial y_{i}}}
	\newcommand{\Dt}{\frac{\partial}{\partial t}}
	\newcommand{\ds}{\displaystyle}
	\newcommand{\Zt}{{\mathcal Z}^{t}}
	\newcommand{\ve}{\varepsilon}
	\newcommand{\D}{\operatorname{div}}
	\newcommand{\G}{\mathscr{G}}
	\newcommand{\w}{\tilde{w}}
	\newcommand{\s}{\sigma}

\title[Sharp order of vanishing, etc.]{Quantitative uniqueness for parabolic equations with H\"older potentials
}

\author{Agnid Banerjee}
\address{School of Mathematical and Statistical Sciences\\ Arizona State University \\ Tempe, AZ 85287-1804\\USA}\email[Agnid Banerjee]{agnid.banerjee@asu.edu}

\author{Nicola Garofalo}
\address{School of Mathematical and Statistical Sciences\\ Arizona State University \\ Tempe, AZ 85287-1804\\USA}\email[Nicola Garofalo]{nicola.garofalo@asu.edu}



\medskip

\maketitle
	
\selectlanguage{english}

\begin{abstract}
In this note we  derive a space-like quantitative uniqueness result for parabolic operators with H\"older zero-order term that interpolates between the Donnelly-Fefferman and the Bourgain-Kenig estimate.  This generalizes a recent result of Teng, Wang and Zhu for the time-independent Schr\"odinger operator with a H\"older potential.

\end{abstract}

\tableofcontents

\section{Introduction}

In a recent work, Teng, Wang and Zhou \cite[Theor. 1 \& Cor. 1]{TWZ} established the following interesting quantitative vanishing-order estimate. Let $u$ be a solution in a ball $B_{10}$ to the equation 
\begin{equation}\label{lap}
\Delta u = V(x) u, 
\end{equation}
where the potential satisfies $||V||_{C^{0,\beta}} \le M$ for some $\beta\in (0,1)$. Assume that $||u||_{L^\infty(B_{10})}\le C_0$ and $||u||_{L^\infty(B_1)}\ge 1$. Then the vanishing-order of $u$ in $B_1$ is at most $CM^{\frac{2}{\beta+3}}$, where $C$ is a constant which depends on $n, \beta$ and $C_0$.

We recall that, when $\Delta = \Delta_M$ is the Laplacian on a compact real-analytic manifold, and $V(x) \equiv -\la$,  Donnelly and Fefferman proved in \cite{DF1, DF2} that the sharp order of vanishing of $u$ is $\cong \sqrt \la$. On the other hand, Bourgain and Kenig proved in \cite{BK} that when $||V||_{L^\infty(B_{10})} \le M$, the maximal order of vanishing is bounded by $C M^{2/3}$. This result is sharp when $V$ is complex-valued in view of Meshkov's example in \cite{Me}. 

Interestingly, the exponent $M^{\frac{2}{\beta+3}}$ for H\"older potentials in \cite{TWZ} 
interpolates between $C^1$ and $L^\infty$ potentials. In other words, it coincides with the  $ M^{1/2}$ Donnelly-Fefferman type vanishing-order estimate for $C^{1}$ potentials when $\beta=1$ (for this, see also \cite{Bk}),
and with the $M^{2/3}$ Bourgain-Kenig estimate when $\beta = 0.$

The present work has been inspired by the cited result in \cite{TWZ}. We establish an optimal space-like order of vanishing for solutions to the following class of parabolic equations in the space-time cylinder $Q_4$
\begin{align}\label{meq}
Lu = u_t+\D(A(x,t)\n u)  + V(x,t)u=0.
\end{align}
The coefficient matrix $A$ satisfies the following assumptions:
\begin{itemize}
\item[(i)] there exist $\Lambda \geq 1$  such that for all $x\in \Rn$ and $t\in [0,\infty)$, one has for every $\xi\in \Rn$
\begin{equation}\label{ell}
\Lambda^{-1} |\xi|^2 \leq \sa A(x,t)\xi,\xi\da\leq \Lambda |\xi|^2\ ;
\end{equation}
\item[(ii)] there exist $C\geq 0$ such that for all $x, y\in \Rn$ and $s, t\in [0,\infty)$, one has
\begin{equation}\label{ass}
|a_{ij}(x,t) - a_{ij}(y,s) |\leq C(|x-y| + |t-s|^{1/2}).
\end{equation}
\end{itemize}
On the zero-order term $V$ we assume that for some $\beta \in (0,1)$, one has
\begin{equation}\label{vassump}
||V||_{H^{\beta}(Q_4)} := ||V||_{L^{\infty}(Q_4)} + \sup_{ \{(x, t), (y, s) \in Q_4: (x, t) \neq (y, s)\}} \frac{|V(x, t) - V(y, s)|}{|x-y|^{\beta} + |t-s|^{\beta/2}}< \infty. \end{equation}
We let 
\begin{equation}\label{lm}
||V||_{H^{\beta}(Q_4)} = M,
\end{equation}
and, without loss of generality, we assume that $M \geq 1$. On a solution to \eqref{meq} we introduce the following quantity, whenever it is well-defined:
\begin{equation}\label{theta}
\Theta := \frac{\int_{Q_4} u^2(x,t) dx dt}{\int_{B_1} u^2(x, 0) dx}.
\end{equation}
The following is our main result.

\begin{thrm}[Space-like vanishing order]\label{main}
Let $u$ be a solution to \eqref{meq} in $Q_4$ such that $u(\cdot, 0) \not \equiv  0$ in $B_1$. There exists a universal constant $N>0$ such that, with
\begin{equation}\label{K}
\mathcal{K}= N\operatorname{log}(N\Theta)+N M^{\frac{2}{\beta+3}},
\end{equation} 
one has for all $r \le 1/2$:  
\begin{align}\label{df}
\int_{B_r} u^2(x,0)dx \geq   r^{\mathcal{K}} \int_{B_1} u^2(x,0)dx.
\end{align}
\end{thrm}

Theorem \ref{main} shows that the space-like vanishing order of a non-trivial solution to \eqref{meq} is upper bounded by $C M^{\frac{2}{\beta+3}}$. We note that, if $u = u(x)$ is a solution to \eqref{lap}, then by letting $\tilde u(x,t) = u(x)$, $\tilde V(x,t) = V(x)$, we can apply Theorem \ref{main} to $\tilde u, \tilde V$, with $A(x,t) \equiv \mathbb I_n$, and recover the above cited result in \cite{TWZ}.

Since, as we have mentioned, our paper is inspired by their work, we provide a brief description and comparison of our approach versus theirs. The key idea in \cite{TWZ} is based on first  splitting the potential $V$ as
\begin{equation}\label{split}
V= V_\ve + (V-V_\ve), 
\end{equation}
where $V_\ve$ is a standard mollification \eqref{mol} of  $V$, and observing that the estimate \eqref{est1} hold.
The authors then implement a nice weighted variant of the  approach in \cite{GL1, GL2} introduced by Kukavica  in \cite{Ku2}. By combining ideas in \cite{D, Zhu1} with the estimates \eqref{est1}, and finally optimizing the choice of $\ve$, they establish their quantitative $CM^{\frac{2}{\beta+3}}$ uniqueness result. 

In this note we show that, after splitting the zero-order term $V(x,t)$ as in \eqref{split},  the parabolic quantitative Carleman estimates derived in our joint works with Arya \cite{AB, ABG} provide a relatively short proof of our Theorem \ref{main}.

For completeness, we provide the reader with a list of various other works on quantitative uniqueness for elliptic and parabolic equations: \cite{Bk, BCa, BG2, CK,  D, D1,  EV,  KSW, KW,  Ku, Ku2, KL, LW,  Zhu1, Zhu2}.

The paper is organised as follows. In Section \ref{s:n}, we introduce the relevant notations and collect some preliminary results that are used in the proof of Theorem \ref{main}. In Section \ref{s:m} we prove such result.

\section{Notations and Preliminaries}\label{s:n}
In this section we introduce the relevant notation and collect some results that will be used in the Proof of Theorem \ref{main} in Section \ref{s:m}.
Points in $\Rn$ will be denoted by $x, y$, etc. For those in space-time $\Rn \times [0, \infty)$, we will use the notation  $X=(x,t), Y=(y,s)$, etc. Whenever convenient, the partial derivative $\partial_{x_i} f$  will be denoted by $f_i$ or $D_i f$, that in $t$ will be denoted by $f_t$ or by $\p_t f$. We will often write $\nabla f$ and  $\operatorname{div} f$ instead of  $\nabla_x f$ and $ \operatorname{div}_x f$, respectively.  

We recall that the \emph{vanishing order} of a function $u$ at $x$ is the largest  integer $\ell$ such that $D^{\alpha} u (x)=0$ for all $|\alpha| \leq \ell$. Here, with $\mathbb N_0 = \mathbb N \cup \{0\}$, we have denoted by $\alpha=(\alpha_1, ..., \alpha_n)\in \mathbb N_0^n$ a multi-index, and have let  $D^\alpha u(x) = \frac{\p^{\alpha_1+...+\alpha_n} u}{\p x_1^{\alpha_1}...\p x_n^{\alpha_n}}(x)$. Given an open set $\Om\subset \Rn\times \R$,
we indicate with $C_0^{\infty}(\Omega)$ the set of functions in $C^\infty(\Om)$ having compact support in $\Omega$. Also, we will indicate by $dX = dxdt$ the Lebesgue measure in $\Rn\times \R$. We will denote by $B_r(x)$ the ball of radius $r$ centred at $x \in \R^n$, whereas $Q_r(x,t)$ will denote the space-time  cylinder $B_r(x) \times [t,t+r^2]$. When $x = 0$ and $t=0$ we will simply write $B_r$ and $Q_r$, instead of $B_r(0)$ and $Q_r(0,0)$, respectively.

Before stating the main Carleman estimate from \cite{ABG} that we use in this paper, we introduce some preparatory ingredients and  relevant results which are collected in Section 2 of that paper.  In the order, they correspond to \cite[Lemmas 4 \& 5]{EF}, and to \cite[Lemmas 3 \& 4]{EFV}. We consider the function $\theta:(0,1)\to \R^+$ defined by 
\begin{equation}\label{theta}
\theta(t)=t^{1/2}\left( \operatorname{log}\frac{1}{t}\right)^{3/2}.
\end{equation} 
This function satisfies the assumptions
\[
0 \le \theta(t) \le N,\ 0\le t \le 1,\ \ \ \ \ \text{and}\ \ \ \ \int_0^1 \left(1+\log \frac 1t\right) \theta(t) \frac{dt}t \le N.
\]
It follows  that there exists $N>0$ such that
\[
|t\theta'(t)|\le N \theta(t),\ \ \ \ \ \ \ \ \ \ t\in (0,1/2].
\]
	
\begin{lemma}\label{sig}
Let $\theta:[0,1/2]\to \R^+$ be as in \eqref{theta}, and for given $\lambda >0$, let $\s:[0,1/2]\to \R$ be the function
\[
\s(t) = t \exp\left[-\int_0^{\la t} \left(1 - \exp\left(-\int_0^s \theta(\tau) \frac{d\tau}\tau\right)\right)\frac{ds}s\right].
\]
Then $\s$ solves the Cauchy problem 
\begin{align}\label{ode}
- \frac{d}{dt}\operatorname{log} \left(\frac{t\s'}{\s}\right) = \frac{d}{dt}\operatorname{log} \left(\frac{\s}{t\s'}\right)=\frac{\theta(\lambda t)}{t},\;\;\;\;\s(0)=0,\;\;\s'(0)=1,
\end{align}
Moreover, there exist a universal constant $N>0$ such that when $0\le \lambda t \le 1/2$:
\begin{itemize}
\item[(i)] $t/N \le \s(t) \le t$;
\item[(ii)]  $1/N \le \s'(t) \le 1$.
\end{itemize}
\end{lemma}
Henceforth in this paper, we will denote by $G$ the function in $\Rn\times (0,\infty)$ thus defined
\begin{equation}\label{G}
G(x,t) := t^{-n/2} e^{-\frac{|x|^2}{4t}}.
\end{equation}
In what follows, with $\s$ defined as in Lemma \ref{sig}, and $G$ as in \eqref{G}, for any given number $a\in [0,1/2]$ we let $\s_a(t)=\s(t+a)$ and  $G_a(x,t)=G(x,t+a)$. One should keep in mind that the domain of $\s_a$ is $[-a,\frac 12 - a].$ By combing the proof of Theorem 3.1 and Lemma 6.1 in \cite{ABG}, one obtains the following  quantitative Carleman estimate which will be used in the proof of Theorem \ref{main}.

\begin{thrm}\label{carleman}
Let $A(0,0)=\mathbb I_n$ and $V_1 \in C^1(B_4), V_2 \in L^{\infty}(B_4)$ be time-independent real-valued functions.  There exist universal constants $N >1$ and $\delta \in (0,1)$ such that, if 
\begin{equation}\label{quanta}
\A \ge N\left(1+||V_1||_{C^{1}(B_4)}^{1/2} +||V_2||_{L^{\infty}(B_4)}^{2/3}\right),
\end{equation}  
and $\lambda=\A/\delta^2$, then the following inequality holds for all $w \in C_{0}^{\infty}\left(B_4 \times [0,\frac{1}{4\lambda})\right)$ and $0<a\le \frac{1}{4\lambda}$:	
\begin{align}\label{carl0}
& \A^2 \int_{B_4 \times [0,\frac{1}{4\lambda}) }\s_a^{-2\A}w^2G_adX +\A \int_{B_4 \times [0,\frac{1}{4\lambda}) } \s_a^{1-2\A}|\n w|^2G_adX\\
& \le N \int_{B_4 \times [0,\frac{1}{4\lambda}) } \s_a^{1-2\A}\left[w_t + \D(A(x,t) \n w)+ V_1(x) w+V_2(x)w\right]^2 G_adX
\notag\\
& +N^{2\A}\A^{2\A}\underset{t \ge 0}{\operatorname{sup}}\int_{\Rn} \left[w^2
+|\n w|^2\right] dx
\notag\\
&+\s(a)^{-2\A}\left(-\frac{a}{N}\int_{\Rn} |\n w (x,0)|^2G(x,a)dx + N \A \int_{\Rn} w^2(x,0)G(x,a)dx\right).
\notag
\end{align}  
\end{thrm}

We close this section by stating the above mentioned approximation lemma following \cite{TWZ}. Let $V = V(x,t)$ be as in \eqref{vassump}, \eqref{lm}, and define $V_0(x)= V(x, 0)$.
Let $\rho \in C^{\infty}_{0} (B_1)$ be such that $\int_{\Rn} \rho dx =1$, denote by $\rho_{\ve} (x) :=\frac{1}{\ve^n} \rho(x/\ve)$ the associated approximate identity, and let
\begin{equation}\label{mol}
V_\ve(x)= V_0 \ast \rho_{\ve}(x).
\end{equation}
The following estimate, which uses the $C^{\beta}$ character of $V_0$, can be found in \cite[Lemma 1]{TWZ}:
\begin{equation}\label{est1}
\begin{cases}
||V_{\ve}||_{C^{1}(B_3)} \leq C_1M \ve^{\beta-1},
\\
||V_0-V_{\ve}||_{L^{\infty}(B_3)} \leq C_1M\ve^{\beta},
\end{cases}
\end{equation}
where $C_1>0$ is a universal constant.  

\section{Proof of Theorem \ref{main}}\label{s:m}
From Theorem \ref{carleman}, the following quantitative Carleman estimate follows for H\"older potentials.
\begin{thrm}\label{carleman1}
Suppose that $V$ satisfies \eqref{vassump} and, without loss of generality, assume that $A(0,0)=\mathbb I_n$. Let $||V||_{H^{\beta}(Q_4)}=M$.   There exist universal constants $N >1$ and $\delta \in (0,1)$ such that, if 
\begin{equation}\label{quanta2}
\A \ge N\left(1+ M^{\frac{2}{\beta+3}} \right), 
\end{equation}  
and $\lambda=\A/\delta^2$, then the following inequality holds for all $w \in C_{0}^{\infty}\left(B_4 \times [0,\frac{1}{4\lambda})\right)$ and $0<a\le \frac{1}{4\lambda}$	
\begin{align}\label{carl1}
& \A^2 \int_{B_4 \times [0,\frac{1}{4\lambda}) }\s_a^{-2\A}w^2G_adX +\A \int_{B_4 \times [0,\frac{1}{4\lambda}) } \s_a^{1-2\A}|\n w|^2G_adX\\
& \le N \int_{B_4 \times [0,\frac{1}{4\lambda}) } \s_a^{1-2\A}\left[w_t + \D(A(x,t) \n w)+ Vw\right]^2 G_adX
\notag\\
& +N^{2\A}\A^{2\A}\underset{t \ge 0}{\operatorname{sup}}\int_{\Rn} \left[w^2
+|\n w|^2\right] dx
\notag\\
&+\s(a)^{-2\A}\left(-\frac{a}{N}\int_{\Rn} |\n w (x,0)|^2G(x,a)dx + N \A \int_{\Rn} w^2(x,0)G(x,a)dx\right).
\notag
\end{align}  
\end{thrm}
\begin{proof}
We proceed in two steps:
\begin{itemize}
\item[(I)] We establish \eqref{carl1} with $V(x, t)$ replaced by $V_0(x)= V (x, 0)$.
\item[(II)] We prove \eqref{carl1} for $V(x,t)$.
\end{itemize}

\noindent (I): Write $V_0 = V_1 + V_2$, where for $V_\ve$ as in \eqref{mol}, we have let  $V_1= V_{\ve}$ and $V_2= V_0 -V_\ve$. We note that, by applying \eqref{est1}, we have for any $\e>0$:
\begin{align*}
& 1+||V_1||_{C^{1}(B_4)}^{1/2} +||V_2||_{L^{\infty}(B_4)}^{2/3} \le 1+ C_1M^{1/2} \ve^{\beta-1} + C_1M^{2/3}\ve^{\beta}
 \le \max\{1,C_1\} \left(1 + M^{1/2} \ve^{-\frac{1-\beta}{2}} + M^{2/3} \ve^{\frac{2\beta}{3}}\right).
\end{align*}
The function in the right-hand side is of the type
\[
g(\e) = 1 + A \e^{-\alpha} + B \e^\gamma,
\]
with $\alpha = \frac{1-\beta}2$, $\gamma = \frac{2\beta}3$, $A = M^{1/2}$, $B = M^{2/3}$. The minimum value of $g$ is 
\[
g_{\min} = 1 + \left[\left(\frac{\gamma}{\alpha}\right)^{\frac{\alpha}{\alpha+\gamma}} + \left(\frac{\alpha}{\gamma}\right)^{\frac{\gamma}{\alpha+\gamma}}\right] A^{\frac{\gamma}{\alpha+\gamma}} B^{\frac{\alpha}{\alpha+\gamma}}.
\]
Substituting values, and adjusting the choice of $N$, we see that for some $N = N(\beta)>1$ we have
\[
1+||V_1||_{C^{1}(B_4)}^{1/2} +||V_2||_{L^{\infty}(B_4)}^{2/3} \le N(1+M^{\frac{2}{\beta+3}}).
\]

From these considerations we conclude that, if $\alpha$ satisfies \eqref{quanta2}, then it verifies \eqref{quanta}.
We are thus in a position to apply 
the Carleman estimate in Theorem \ref{carleman} and obtain
\begin{align}\label{carl2}
& \A^2 \int_{B_4 \times [0,\frac{1}{4\lambda}) }\s_a^{-2\A}w^2G_adX +\A \int_{B_4 \times [0,\frac{1}{4\lambda}) } \s_a^{1-2\A}|\n w|^2G_adX\\
& \le N \int_{B_4 \times [0,\frac{1}{4\lambda}) } \s_a^{1-2\A}\left[w_t + \D(A(x,t) \n w)+ V_0w\right]^2 G_adX
\notag\\
& +N^{2\A}\A^{2\A}\underset{t \ge 0}{\operatorname{sup}}\int_{\Rn} \left[w^2
+|\n w|^2\right] dx
\notag\\
&+\s(a)^{-2\A}\left(-\frac{a}{N}\int_{\Rn} |\n w (x,0)|^2G(x,a)dx + N \A \int_{\Rn} w^2(x,0)G(x,a)dx\right).
\notag
\end{align} 
This proves \eqref{carl1} with $V_0(x)$ in place of $V(x, t)$, thus completing step (I).
 
\medskip

\noindent (II): We finally want to obtain \eqref{carl1} for $V(x,t)$. With this goal in mind, we observe that by the Cauchy-Schwarz and triangle inequalities, we have
\begin{align*}
& |(w_t + \D(A(x,t) \n w)+V(x,0)w )|^2
\\
& =  |w_t+ (\D(A(x,t) \n w)+V(x,t)w )+(V(x,0)-V(x,t))w |^2
\\
& \le 2(w_t + \D(A(x,t) \n w)+V(x,t)w)^2+2(V(x,0)-V(x,t))^2w^2.
\end{align*}
On the other hand, \eqref{vassump} gives 
\begin{equation*}
(V(x,0)-V(x,t))^2 \le M^{2} t^{\beta}.
\end{equation*}
Using the latter two inequalities, and keeping in mind that $V_0(x) = V(x,0)$, we find 
\begin{align}\label{v10}
& N \int \s_a^{1-2\A}\left[w_t + \D(A(x,t) \n w)+w_t+V_0 w\right]^2 G_adX
\\ & \le  2N \int \s_a^{1-2\A}\left[w_t+ \D(A(x,t) \n w) +V(x,t)w\right]^2 G_adX
 \notag
 \\
 &+ 2N M^2 \int \s_a^{1-2\A}t^{\beta}w^2 G_adX.
 \notag
\end{align}
We would be done if  we could absorb the last  term in the right-hand side of \eqref{v10}  into  the left-hand side of \eqref{carl2}. Since $\lambda t \le 1/2$  and $\s_a(t) \le t+a \le 1/\lambda$, we have
\begin{align}\label{kv1}
&2N M^2 \int \s_a^{1-2\A}t^{\beta}w^2 G_adX.\\ &   \le 	\frac{ 2N}{2^\beta \lambda^{1+\beta}} M^2 \int \s_a^{-2\A}w^2 G_a dX.\notag
\end{align}
Inserting \eqref{v10} and \eqref{kv1} in \eqref{carl2}, and keeping in mind that, for $\delta \leq \frac{1}{\sqrt{2}}$, we have $\alpha \leq \frac{\A}{2\delta^2} = \frac{\lambda}{2}$, we obtain for a new $N$ that the following estimate holds
\begin{align}\label{f2}
&	\A^2 \int \s_a^{-2\A}w^2G_a dX +\A \int \s_a^{1-2\A}|\n w|^2G_adX\\
& \le N \int \s_a^{1-2\A}(w_t + \D(A(x,t) \n w)+V(x,t)w)^2 G_adX\notag\\
&+\frac{ N}{\A^{1+\beta}} M^2 \int \s_a^{-2\A}w^2 G_adX +N^{2\A}\A^{2\A}\underset{t \ge 0}{\operatorname{sup}}\int [w^2
+|\n w|^2] dx\notag\\
&+\s(a)^{-2\A}\left(-\frac{a}{N}\int |\n w (x,0)|^2G(x,a)dx + N \A \int w^2(x,0)G(x,a)dx\right).\notag
\end{align}
Now observe that the second term in the right-hand side of \eqref{f2} can be absorbed in the first  term of left-hand side of \eqref{f2}, provided that 
\begin{align}\label{last}
\frac{\A^2}{2} \ge \frac{ N}{\alpha^{1+\beta}} M^2 \Longleftrightarrow\  \A  \geq (2N)^{\frac{1}{\beta+3}} M^{\frac{2}{\beta+3}}.
\end{align}
Since we are imposing on $\alpha$ that it satisfy \eqref{quanta2}, it is clear that we can achieve \eqref{last}. Tracing back our steps, we obtain \eqref{carl1} for $V(x,t)$. This completes the proof.

\end{proof}

The last ingredient we need is the following quantitative monotonicity in time result. Before stating it, we recall the following quantities from \cite{ABG} defined on a solution $u$ to \eqref{meq} in $Q_4$ we define for any $\rho\in (0,1]$: 
\[
\Theta_{\rho} :=\frac{\int_{Q_{3}} u^2( x, t)dX}{\rho^2 \int_{B_{\rho}}u^2( x,0)dx},
\]
and 
\[
T_{\rho,N} := \frac{\rho^2}{2N\operatorname{log}(2N(1+||V||_{\infty})^2{\Theta_{\rho}})+ 5 N^2 (||V||_{\infty}^{1/2}+1)}.
\]

\begin{lemma}\label{mon}
Let $u$ be a solution to \eqref{meq} in $Q_4$. Then there exists a universal constant $N>1$ such that, for every $\rho\in (0,1]$ and $t\le T_{\rho,N}$, the following inequality holds
\begin{align}\label{mon1}
Ne^{ M^{\frac{2}{\beta+3}}} \int_{B_{2\rho}} u^2(x,t)dx \ge \int_{B_{\rho}} u^2(x,0)dx.
\end{align}
\end{lemma}

\begin{proof}
In \cite[Lemma 4.1]{ABG} the following estimate was proven for $t \leq T_{\rho, N}$ and $N$ large
\begin{equation}\label{mo2}
Ne^{ (||V||_{\infty}^{1/2}+1)}\int_{B_{2\rho}} u^2(x,t)dx \ge \int_{B_{\rho}} u^2(x,0)dx.\end{equation}
Using $||V||_{\infty} \leq M$, with the normalization $M\geq 1$,  and the fact that $\frac{2}{\beta+3} > \frac{1}{2}$ for $\beta \in (0, 1)$, we find that the desired estimate \eqref{mon1} follows from \eqref{mo2}. 

\end{proof}

Finally, we now turn to

\begin{proof}[Proof of Theorem \ref{main}]
With Theorem \ref{carleman1} and Lemma \ref{mon} in hand, both of which  have quantitative dependence on the exponent $M^{\frac{2}{\beta+3}}$, one can repeat the arguments in the proof of either \cite[Theorem 1]{AB} or \cite[Theorem 1.1]{ABG} to assert that the quantitative doubling inequality \eqref{df} in Theorem \ref{main} holds.

\end{proof}

\end{document}